\documentclass[11pt,leqno]{amsart}
\usepackage{amsmath,amssymb,amsthm}
\usepackage{hyperref}
\newtheorem{theorem}{Theorem}[section]
\newtheorem{lemma}[theorem]{Lemma}
\newtheorem{ques}[theorem]{Question}

\newtheorem{proposition}[theorem]{Proposition}
\newtheorem{corollary}[theorem]{Corollary}
\theoremstyle{definition}
\newtheorem{definition}[theorem]{Definition}

\theoremstyle{remark}
\newtheorem{remark}[theorem]{Remark}
\numberwithin{equation}{section}
\def\fnote#1{\footnote}

\def\ignora#1{}
\def\n3#1{\left\vert  \! \left\vert \! \left\vert \, #1 \, \right\vert \!
  \right\vert \! \right\vert }

\DeclareMathOperator{\dmn}{dim\,}

\renewcommand{\geq}{\geqslant}
\renewcommand{\leq}{\leqslant}

\newcommand{\iten}{\ensuremath{\widehat{\otimes}_\varepsilon}}
\newcommand{\pten}{\ensuremath{\widehat{\otimes}_\pi}}

\begin{document}

\title[SSD2P in tensor products]{ Symmetric strong diameter two property in tensor products of Banach spaces }

\author {Johann Langemets} 

\address{Institute of Mathematics and Statistics, University of Tartu, Narva mnt 18, 51009 Tartu, Estonia}
\email{johann.langemets@ut.ee}
\thanks{This work was supported by the Estonian Research Council grant (PSG487).}
\urladdr{\url{https://johannlangemets.wordpress.com/}}


\begin{abstract}
	We continue the investigation of the behaviour of diameter two properties in tensor products of Banach spaces. Our main result shows that the symmetric strong diameter two property is stable by taking projective tensor products. We also prove a result for the symmetric strong diameter two property for the injective tensor product.
\end{abstract}

\maketitle

\section{Introduction}
In the last 15 years a lot of attention has been devoted to various diameter two properties since several classical Banach spaces failing the Radon-Nikod\'ym property have them (see \cite{ALL}--\cite{becerra_guerrero_octahedral_2015}, \cite{hlln}--\cite{ruedaarxiv}). 

According to \cite{ANP} a Banach space $X$ has the
	\emph{symmetric strong diameter 2 property} (SSD2P) whenever $n\in \mathbb N$, $S_1,\dots, S_n$ are slices of $B_X$ and
	$\varepsilon > 0$, there exist $x_i \in S_i$ and $y \in B_X$,
	independent of $i$, such that $x_i \pm y \in S_i$ for every
	$i \in \{1,\dots,n\}$ and $\|y\| > 1 - \varepsilon$.

Geometrically the SSD2P of a Banach space says that, given a finite number of slices of the unit ball, there exists a direction such that all these slices contain a line segment of length almost 2 in this direction. Banach spaces with the SSD2P include, for example, uniform algbras, infinite-dimensional preduals of $L_1$-spaces, $M$-embedded spaces, and even certain Lipschitz spaces (see \cite{hlln} and \cite{LRZ}). If a Banach
space has the SSD2P, then the space has the strong diameter two property (SD2P), that is, every convex combination of slices of the unit ball has diameter two \cite[Lemma~4.1]{ALN}. Observe that the SSD2P is strictly stronger than
the SD2P, for example, $L_1[0,1]$ has the SD2P, but not the SSD2P \cite[Remark~3.3]{hlln}. In \cite{hlln} the following useful characterization of the SSD2P was obtained.

\begin{theorem}[see {\cite[Theorem 2.1]{hlln}}]\label{thm: SSD2P char}
	Let $X$ be a Banach space. The following assertions are equivalent:
	\begin{enumerate}
		\item[(i)]\label{item:ssd2p-char-a}
		$X$ has the SSD2P.
		\item[(ii)]\label{item:ssd2p-char-b}
		Whenever $n \in \mathbb{N}$, $U_1, \ldots, U_n$
		are nonempty relatively weakly open subsets of $B_X$
		and $\varepsilon > 0$,
		there exist $x_i \in U_i$, $i \in \{1,\dots,n\}$,
		and $y \in B_X$ such that $x_i \pm y \in U_i$ for
		every $i \in \{1,\dots,n\}$ and $\|y\| > 1 - \varepsilon$.
		\item[(iii)]\label{item:ssd2p-char-d}
		Whenever $n \in \mathbb{N}$,
		$x_1, \dots, x_n \in S_X$,
		there exist nets
		$(y^i_\alpha) \subset S_X$ and
		$(z_\alpha) \subset S_X$ such that
		$y^i_\alpha \to x_i$ weakly, $z_\alpha \to 0$
		weakly, and $\|y^i_\alpha \pm z_\alpha\|\rightarrow 1$
		for every $i \in \{1,\dots,n\}$.
	\end{enumerate}
\end{theorem}

Recall from \cite{ALL} that a Banach space $X$
is \emph{almost square} (ASQ) if whenever $n \in \mathbb{N}$,
and $x_1,\dotsc,x_n \in S_X$, there exists a sequence
$(y_k) \subset S_X$ such that $y_k \to 0$
weakly and $\|x_i \pm y_k\| \to 1$ for every $i \in \{1,\dots,n\}$. Hence, by Theorem~\ref{thm: SSD2P char} (iii), one has that an ASQ Banach space always has the SSD2P. However, the
converse fails, $C[0,1]$ has the SSD2P (see the proof of \cite[Proposition~4.6]{ALN2})
and is not ASQ (this can be easily seen by considering the
constant 1 function).  Spaces which are ASQ include $c_0(X_n)$,
where $X_n$ are arbitrary Banach spaces,
and Banach spaces $X$ which are M-ideals
in $X^{**}$ (see \cite{ALL}).

To summarize, we have the following diagram
\[
\text{ ASQ }\Rightarrow \text{ SSD2P } \Rightarrow \text{ SD2P,}
\]
where none of the implications is reversible.

Abrahamsen, Lima, and Nygaard asked in \cite[Section 5, (b)]{ALN} how are the diameter two properties in general preserved by tensor products. In \cite[Theorem~3.5]{becerra_guerrero_octahedral_2015} Becerra~Guerrero, {L}\'{o}pez-{P}\'{e}rez, and Rueda~Zoca proved that the SD2P is preserved from both factors by taking projective tensor product of Banach spaces. Actually, one can even weaken the hypothesis on one of the factors \cite[Theorem~2.2]{hlp}. Very recently Rueda Zoca showed that almost squareness is also preserved from both factors by taking projective tensor product \cite[Theorem~2.1]{rueda}. Therefore it is natural to wonder whether the SSD2P is also stable by forming projective tensor products of Banach spaces. In Section~\ref{sec: projective tensor} we will prove that it is indeed so (see Theorem~\ref{thm: SSD2P in proj tensor product}). This result is proven by making use of the equivalent characterization of the SSD2P from Theorem~\ref{thm: SSD2P char} and classical Rademacher techniques (see \cite{rueda} and \cite{rya}). In Section~\ref{sec: inj tensor} we will provide a sufficient condition on $X$ to assure the SSD2P in $X\iten Y$ for any nontrivial Banach space $Y$.

We pass now to introduce some notation. 
All Banach spaces considered in this paper are nontrivial and over
the real field. The closed unit ball of a Banach space $X$ is denoted
by $B_X$ and its unit sphere by $S_X$. The dual space of $X$ is
denoted by $X^\ast$ and the bidual by $X^{\ast\ast}$.
By a \emph{slice} of $B_X$ we mean a set of the form
\begin{equation*}
S(B_X, x^*,\alpha) :=
\{
x \in B_X : x^*(x) > 1 - \alpha
\},
\end{equation*}
where $x^* \in S_{X^*}$ and $\alpha > 0$.

Given two Banach spaces $X$ and $Y$, we will denote by $X\pten Y$ the projective and by $X\iten Y$ the injective tensor product of $X$ and $Y$. Recall that the space $\mathcal{B}(X\times Y)$ of bounded bilinear forms defined on $X\times Y$ is linearly isometric to the topological dual of $X\pten Y$. We refer to \cite{rya} for a detailed treatment and applications of tensor products.

For a Banach space $X$ we denote by $\mathcal{L}(X)$ the space of all bounded and linear operators on $X$. By a \emph{multiplier} on $X$ we mean an element $T\in \mathcal{L}(X)$ such that every extreme point of $B_{X^*}$ becomes an eigenvector for $T^*$. The \emph{centralizer} of a real Banach space $X$ (denoted by $Z(X)$) coincides with the set of all multipliers on $X$. We refer to \cite{Behrends} for a detailed treatment of centralizers.

\section{Projective tensor product}\label{sec: projective tensor} 

We begin by recalling a slight rephrasing of a Lemma from \cite{rueda}, which we will include without a proof for the sake of readability.

\begin{lemma}[see {\cite[Lemma~2.2]{rueda}}]\label{lemma: rademacher} Let $X$ and $Y$ be Banach spaces. If  $x,\tilde{x}\in B_X$ and $y,\tilde{y}\in B_Y$ are such that 
	\[
	\|x\pm \tilde{x}\|\leq 1\text{ and } \|y\pm \tilde{y}\|\leq 1,
\]
then 
	\[
	\|x\otimes y \pm \tilde{x}\otimes \tilde{y}\|\leq 1.
	\]
\end{lemma}

We are now ready to prove our main result as promised in the Introduction, which will provide a large class of Banach spaces with the SSD2P.

 \begin{theorem}\label{thm: SSD2P in proj tensor product}
 Let $X$ and $Y$ be Banach spaces. If $X$ and $Y$ have the SSD2P, then so does $X\pten Y$.
 \end{theorem}
 \begin{proof}
 Let $n\in\mathbb{N}$ and consider the slices $S_i:=S(B_{X\pten Y}, B_i, \alpha_i)$, $i\in \{1,\dots,n\}$, where $B_i$ are norm one bilinear forms and $\alpha_i>0$. Let $\varepsilon>0$ be such that $\varepsilon<\min_{i\in\{1,\dots,n\}} \alpha_i$. We will show that there are $z_i\in S_i$ and $z\in B_{X\pten Y}$ such that $z_i\pm z\in S_i$ for every $i\in \{1,\dots,n\}$ and $\|z\|>1-\varepsilon$.
 
 Let $\delta>0$. Choose elements $u_i\otimes v_i\in S_X\otimes S_Y$ such that $B_i(u_i,v_i)>1-\delta$ for every $i\in \{1,\dots,n\}$. 
 
 Consider first the slices $U_i:=S(B_X, \frac{B_i(\cdot, v_i)}{\|B_i(\cdot, v_i)\|}, \delta)$ of $B_X$. Since $X$ has the SSD2P, there are $x_i\in U_i$ and $x\in B_X$ such that $x_i\pm x\in U_i$ and $\|x\|> 1-\delta$. Therefore,
 \[
 |B_i(x, v_i)| <\delta\|B_i(\cdot, v_i)\| \quad \text{for every $i\in \{1,\dots,n\}$,}
 \]
 because $\|x_i\pm x\|\leq 1$ and $B_i(x_i,v_i)>(1-\delta)\|B_i(\cdot, v_i)\|$. Consider now the relatively weakly open sets 
 \[
 V_i:=\{y\in B_Y\colon B_i(x_i, y)>(1-\delta)\|B_i(\cdot,v_i)\|\}
 \]
 and
 \[V^0_i:=\{y\in B_Y\colon |B_i(x, y)| <\delta\|B_i(\cdot, v_i)\|,\quad i\in \{1,\dots,n\} \}.
\]
Observe that $W_i:= V_i\cap V^0_i$ are nonempty relatively weakly open subsets of $B_Y$, because $v_i\in W_i$ for every $i\in \{1,\dots,n\}$. Since $Y$ also has the SSD2P, by Theorem~\ref{thm: SSD2P char}~(ii), there are $y_i\in W_i$ and $y\in B_Y$ such that $y_i\pm y\in W_i$ and $\|y\|> 1-\delta$. Observe that
\[
|B_i(x, y)| <2\delta \|B_i(\cdot, v_i)\| \quad \text{for every $i\in \{1,\dots,n\}$,}
\]
because $|B_i(x, y_i\pm y)| <\delta \|B_i(\cdot, v_i)\|$  and $|B_i(x, y_i)| <\delta \|B_i(\cdot, v_i)\|$.

For every $i\in \{1,\dots,n\}$ set $z_i:=x_i\otimes y_i$ and $z:=x\otimes y$. Clearly, $\|z_i\|\leq 1$  and $(1-\delta)^2\leq\|z\|\leq 1$. The fact that $\|z_i\pm z\|\leq 1$ follows from Lemma~\ref{lemma: rademacher}.

Finally,
\begin{align*}
B_i(z_i)=B_i(x_i,y_i)>(1-\delta)\|B_i(\cdot,v_i)\|>(1-\delta)^2
\end{align*}
and
\begin{align*}
    B_i(z_i\pm z)&=B_i(x_i,y_i)\pm B_i(x,y)\\
    &\geq B_i(x_i,y_i)-|B_i(x,y)|\\
    &>(1-\delta)^2-2\delta \|B_i(\cdot, v_i)\|\\
    &\geq 1-4\delta+\delta^2.
\end{align*}
Therefore, by choosing $\delta$ small enough, we can assure that the elements $z_i$ and $z$ are the ones we need.
 \end{proof}

\begin{remark}
	Theorem \ref{thm: SSD2P in proj tensor product} remains no longer true if one assumes that only $X$ has the SSD2P. Indeed, by \cite[Corollary~3.9]{llr2}, the space $\ell_\infty\pten \ell^3_3$ fails the SD2P (hence also the SSD2P) although $\ell_\infty$ has the SSD2P.
\end{remark}

\begin{remark}
	Recall a property for a Banach space $X$ that was used as part of the hypothesis in \cite[Theorem~3.2]{becerra_guerrero_octahedral_2015}:
	\begin{itemize}
		\item[(P)] there is a $u\in S_X$ such that for every $x\in S_X$ and every $\varepsilon>0$, there is an $x^*\in B_{X^*}$ satisfying $|x^*(x)|>1-\varepsilon\quad\text{and}\quad x^*(u)=1.$
	\end{itemize}
	From \cite[Theorem~3.2]{becerra_guerrero_octahedral_2015} it follows that if $X$ has the SD2P and $Y^*$ has (P), then $X\pten Y$ has the SD2P. However, a similar statement for the SSD2P is no longer true, because $\ell_\infty$ has (P), but $X\pten \ell_1=\ell_1(X)$ which never has the SSD2P \cite[Theorem~3.1]{hlln}.
	
\end{remark}

\begin{remark}\label{rem: sequential SD2P}
	In \cite[Definition~2.6]{rueda} a formally stronger version of the SSD2P was introduced, which is called the sequential SD2P. Following the ideas in the proof of \cite[Theorem~2.1]{rueda} a bit more technical proof gives that the sequential SD2P is also stable by taking projective tensor products.
\end{remark}

For a Banach space $X$ one can consider the increasing sequence of its even duals
\[
X\subset X^{\ast\ast}\subset X^{(4}\subset \dots 
\]
Since every Banach space is isometrically embedded into its second dual, we can define $X^{(\infty}$ as the completion of the normed space $\bigcup_{n=0}^{\infty} X^{(2n}$. 

Observe that the proof of \cite[Theorem~3.4]{ABGLP2012} actually shows that: 
\begin{lemma}\label{lem: infinte centralizer}
	If $\dmn (Z(X^{(\infty}))=\infty$, then $X$ has the SSD2P.
\end{lemma}

\begin{remark}
	Banach spaces which satisfy the assumption of Lemma~\ref{lem: infinte centralizer} include, for example, infinite-dimensional preduals of $L_1$-spaces, infinite-dimensional $C^{*}$-algebras, and $\mathcal{L}(X,Y)$ whenever $Z(Y)$ is infinite-dimensional (see \cite[pp.~466]{ABGRP}). 
\end{remark}

Therefore, by combining Thereom~\ref{thm: SSD2P in proj tensor product} with Lemma~\ref{lem: infinte centralizer}, we get a new consequence, which improves the result \cite[Corollary~3.8]{becerra_guerrero_octahedral_2015}, where under the same hypotheses it is obtained that $X\pten Y$ has the SD2P.

\begin{corollary}
	Let $X$ and $Y$ be Banach spaces. If $Z(X^{(\infty})$ and $Z(Y^{(\infty})$ are infinite-dimensional, then $X\pten Y$ has the SSD2P.
\end{corollary}

\section{Injective tensor product}\label{sec: inj tensor}

Recall from \cite[Corollary~2.7]{llr} that $X\iten Y$ is ASQ whenever $X$ is ASQ. However, it is not known whether a similar result holds for the SD2P \cite[Remark~4.5 (2)]{ruedaarxiv}. 

On the contary to stability of the SSD2P in projective tensor products, we will now prove that for injective tensor products it is enough to assume a sufficient condition on only one of the factors. Next result improves \cite[Theorem~5.3]{ABGRP}, where under the same hypotheses it is obtained that $X\iten Y$ has the diameter two property, however our argument follows essentially the same idea.

\begin{theorem}\label{thm: SSD2P injective tensor}
	Let $X$ and $Y$ be Banach spaces. If $\sup\{\dmn (Z(X^{(2n}))\colon n\in\mathbb N)\}=\infty$, then $X\iten Y$ has the SSD2P.  
\end{theorem}
\begin{proof}
Arguing as in the beginning of the proof of \cite[Theorem~5.3]{ABGRP} one has that 
\begin{equation}\label{eq: inj tensor}
(X\iten Y)^{(\infty}=(X^{(2n}\iten Y)^{(\infty}\quad \text{ for every $n\in \mathbb{N}$.}
\end{equation}

	On the contrary, suppose that $X\iten Y$ fails the SSD2P, then, by Lemma~\ref{lem: infinte centralizer}, there is a $m\in \mathbb{N}$ such that $\dmn (Z((X\iten Y)^{(\infty}))\leq m$. Therefore, by (\ref{eq: inj tensor}), for every $n\in \mathbb N$ we have that $\dmn (Z((X^{(2n}\iten Y)^{(\infty}))\leq m$ also. Since $Z(X^{(2n}\iten Y)$ contains a copy of $Z(X^{(2n})\otimes Z(Y)$ (see \cite{Wickstead}) and $Z(Y)\neq 0$ we conclude that $\dmn Z(X^{(2n})\leq m$ for every $n\in \mathbb{N}$. This contradicts the assumption that $\sup\{\dmn (Z(X^{(2n}))\colon n\in\mathbb N)\}=\infty$.
\end{proof}

In the light of Theorem~\ref{thm: SSD2P injective tensor} and \cite[Corollary~2.7]{llr} it is natural to wonder.

\begin{ques}
	If $X$ has the SSD2P, then $X\iten Y$ has the SSD2P?
\end{ques}

\section*{Acknowledgments}
The author wishes to thank Abraham Rueda Zoca for his comments on the topic of this paper, in particular, for pointing out Remark~\ref{rem: sequential SD2P}.

\end{document}